\documentclass[11pt,a4paper,reqno]{amsart}

\usepackage{amsfonts}
\usepackage{amsmath}
\usepackage{amssymb}
\usepackage{amsthm}
\usepackage{amsxtra}
\usepackage{bbm}
\usepackage{braket}
\usepackage{comment}
\usepackage{extarrows}
\usepackage{latexsym}
\usepackage{mathabx}
\usepackage{mathrsfs}
\usepackage{mathtools}
\usepackage{verbatim}

\usepackage[pdfborder={0 0 0}]{hyperref} 
\hypersetup{breaklinks, raiselinks}

\usepackage{bm}

\usepackage[initials,lite]{amsrefs} 
\AtBeginDocument{%
    \def\MR#1{}
}

\usepackage{a4wide}
\theoremstyle{plain}
\newtheorem{Theorem}{Theorem}[section]
\newtheorem{Lemma}[Theorem]{Lemma}
\newtheorem{Corollary}[Theorem]{Corollary}
\newtheorem{Proposition}[Theorem]{Proposition}

\theoremstyle{definition}

\newtheorem{Assumptions and Discussion}[Theorem]{Assumptions and Discussion}

\newtheorem{Example}[Theorem]{Example}

\newtheorem{Remark}[Theorem]{Remark}

\theoremstyle{remark}

\newtheorem*{acknowledgment*}{Acknowledgment}
\newtheorem{acknowledgements}[Theorem]{Acknowledgements}

\usepackage[shortlabels]{enumitem}
\SetEnumerateShortLabel{a}{\textup{(\alph*)}}
\SetEnumerateShortLabel{A}{\textup{(\Alph*)}}
\SetEnumerateShortLabel{1}{\textup{(\arabic*)}}
\SetEnumerateShortLabel{i}{\textup{(\roman*)}}
\SetEnumerateShortLabel{I}{\textup{(\Roman*)}}

\def\deg{\operatorname{deg}}
\def\depth{\operatorname{depth}}

\def\dim{\operatorname{dim}}

\def\Gr{\operatorname{Gr}}

\def\H{\operatorname{H}} 

\def\ini{\operatorname{in}} 

\def\into{\hookrightarrow}

\def\ker{\operatorname{ker}}
\def\KK{{\mathbb K}}

\def\NN{{\mathbb N}}

\def\PP{{\mathbb P}}

\def\red{\operatorname{red}}

\def\reltype{\operatorname{reltype}} 

\def\spoly{\operatorname{spoly}}

\def\Sym{\operatorname{Sym}}

\def\ZZ{{\mathbb Z}}

\newcommand\bdalpha{{\bm \alpha}}

\newcommand\bda{{\bm a}}
\newcommand\bdbeta{{\bm \beta}}

\newcommand\bdb{{\bm b}}

\newcommand\bdc{{\bm c}}

\newcommand\bddelta{{\bm \delta}}

\newcommand\bde{{\bm e}}

\newcommand\bdgamma{{\bm \gamma}}

\newcommand\bdT{{\bm T}}

\newcommand\bdX{{\bm X}}

\newcommand\bdzero{{\bm 0}}

\newcommand\calC{\mathcal{C}}

\newcommand\calF{\mathcal{F}}

\newcommand\calH{\mathcal{H}}

\newcommand\calP{\mathcal{P}}

\newcommand\calR{\mathcal{R}}

\newcommand\frakM{\mathfrak{M}}
\newcommand\frakm{\mathfrak{m}}

\newcommand\rme{\mathrm{e}}

\newcommand\Spec{\operatorname{Spec}}

\newcommand{\Proj}{\operatorname{Proj}}

\def\reg{\operatorname{reg}}

\begin{document}

\title[On the conjecture of Vasconcelos]{On the conjecture of Vasconcelos for Artinian almost complete intersection monomial ideals}

\author[Kuei-Nuan Lin, Yi-Huang Shen]{Kuei-Nuan Lin and Yi-Huang Shen}

\thanks{AMS 2010 {\em Mathematics Subject Classification}.
    13A30, 
    13D02, 
    13H10, 
    13P10  
}


\address{Penn State Greater Allegheny, Department of Mathematics, McKeesport, PA, USA}
\email{kul20@psu.edu}

\address{Key Laboratory of Wu Wen-Tsun Mathematics, Chinese Academy of Sciences, School of Mathematical Sciences, University of Science and Technology of China, Hefei, Anhui, 230026, P.R.~China}
\email{yhshen@ustc.edu.cn} 

\maketitle

\begin{center}{
        {\em \small Dedicated to our advisor, Professor Bernd Ulrich,\\
            for the long-lasting advising, patience and friendship.}
    }
\end{center}

\begin{abstract}
    In this short note, we confirm a conjecture of Vasconcelos which states that the Rees algebra of any Artinian almost complete intersection monomial ideal is almost Cohen--Macaulay. 
\end{abstract}

\section{Introduction}

Let $R:=\KK[T_1,\dots,T_m]$ be a polynomial ring in the variables $T_1,\dots,T_m$ over a field $\KK$ with $m\ge 2$. In 2013, Vasconcelos \cite[Conjecture 4.15]{MR3096902} conjectured that the Rees algebra of any Artinian almost complete intersection monomial ideal in $R$ is almost Cohen--Macaulay. Recently, supporting proofs in the binary case were established in  \cite{MR3320795}, \cite{MR2013172} and \cite{MR3295062} using different techniques. With the help of Sylvester forms, similar results were also obtained for uniform ideals in \cite{MR3490460} and \cite{MR3295062}. The aim of the current note is to confirm this conjecture completely.  

In general, equi-generated almost complete intersection ideals play an important role in the elimination theory of parameterizations and their Rees algebras encapsulate some of the most common tools for that purpose. Recall that the Rees algebra of the ideal $I\subset R$ is 
\[
\calR(I)=R[IZ]:=\bigoplus_{k=0}^{\infty} I^k Z^k \subset R[Z],
\]
where $Z$ is a new variable.  As the scheme $\Proj(\calR(I))$ is the blowup of $\Spec(R)$ along $V(I)$, $\calR(I)$ plays an important role in the birational study of algebraic varieties. The study of Rees algebras is wide-open.  Even in the binary parameterizations cases, details of the Rees algebra $\calR(I)$ is still unknown in general; see for example \cite{MR2770555}. 

From our algebraic point of view, the Cohen--Macaulayness is a desirable property that a nice algebra is expected to retain.  However, the Rees algebra is typically not Cohen--Macaulay. Thus, it is very natural to seek the next to best case and the notion of almost Cohen--Macaulayness arises.  The standard graded algebra $\calR(I)$ is called \emph{almost Cohen--Macaulay} if $\depth(\calR(I))\ge \dim(\calR(I))-1$.  Ideally, many of the established tools for studying $\calR(I)$ when this algebra is Cohen--Macaulay still apply when this algebra is only almost Cohen--Macaulay; cf.~\cite{zbMATH06113058} and \cite{MR3096902}.  

To study the aforementioned conjecture of Vasconcelos, let 
\begin{equation}
I=\braket{T_1^{a_1},T_2^{a_2},\dots,T_m^{a_m},T_1^{b_1}T_2^{b_2}\cdots T_{m}^{b_{m}}} 
\label{eqn:I}
\end{equation}
be such an Artinian almost complete intersection monomial  ideal in $R=\KK[T_1,\dots,T_m]$. In the following, we will fix these integral vectors $\bda=(a_1,\dots,a_m)$ and $\bdb=(b_1,\dots,b_m)$. Without loss of generality, we may assume that $0\le b_i<a_i$ for every $i$ and there are at least two different indices $i,j$ for which $b_i\ne 0$ and $b_j\ne 0$. Since the case of $m=2$ is well understood, we will further assume that $m\ge 3$.  

To understand the Rees algebra $\calR(I)$, we can alternatively study the free presentation
\[
0\to L \into S=R[X_1,\dots,X_m,W]\xrightarrow{\Phi} \calR(I)\to 0,
\]
where the $X_1,\dots,X_m,W$ are new indeterminates over $R$, and the graded $R$-algebra homomorphism $\Phi$ is determined by 
\[
X_{i}\mapsto T_{i}^{a_i}Z \quad \text{for $1\le i\le m$, \qquad and}\quad W\mapsto T_1^{b_1}T_2^{b_2}\cdots T_m^{b_m}Z.
\] 
The prime ideal $L=\ker(\Phi)$ is known as the \emph{ideal of equations} of $I$, or the \emph{defining ideal} of the Rees algebra $\calR(I)$.  It encodes the syzygies of all powers of $I$ and is inevitably an vital algebraic tool for understanding the geometric properties of several constructions related to $I$.  It also plays a prominent role in geometry modeling community; see for example \cite{MR2394983}. And in our situation, the almost Cohen--Macaulayness of $\calR(I)$ simply means $\depth_S(S/L)\ge m$. 

As a matter of fact, to attack the almost Cohen--Macaulay conjecture of Vasconcelos, we follow some of the nice ideas in \cite{MR3490460} and \cite{MR3295062}. However, since our ultimate aim is not to give a detailed description of a minimal generating set of the defining ideal $L$, we drop the Sylvester-form approach. Instead, we merely give a profile of an infinite Gr\"obner basis of $L$ in Section 2, which is good enough for us to carry on the discussion. As for the monomial ordering in need, the special attention towards $W$ seems essential. After that, we also study a sequence of sub-ideals of $L$, and the accompanied colon ideals in Section 3. With these preparations and the standard Depth Lemma technique, we are able to confirm the almost Cohen--Macaulay conjecture of Vasconcelos at the end of this section.
In particular, we are able to cover the corresponding almost Cohen--Macaulay results in  \cite[Theorem 2.5]{MR3490460}, \cite[Proposition 4.12]{MR3096902} and \cite[Theorems 3.7 and 3.14]{MR3295062}.  
In the last section, we give further applications towards the defining ideal $L$ and the Rees algebra $\calR(I)$, when $I$ is equi-generated. The help from the established Gr\"obner basis profile of $L$ in Section 2 and the almost Cohen--Macaulayness of $\calR(I)$ in Section 3 is indispensable for that purpose. 

\section{Gr\"obner basis of the defining ideal}
As the first step, we will give an infinite Gr\"obner basis of the defining ideal $L$ in this section. Of course, to make it handy, we will reduce it to a finite subset later. At the end of this section, we will discuss when the reduced (hence finite) Gr\"obner basis of $L$ is indeed a minimal generating set.

Let $\NN$ be the set of non-negative integers $\{0,1,2,\dots\}$. For any vector $\bde\in \NN^m$, we will write $e_i$ for its $i$-th component where $1\le i\le m$, and set $|\bde|:=\sum_{i=1}^m e_i$. As usual, we will write $\bdX$ for the ordered sequence $X_1,\dots,X_m$, and $\bdT$ for the ordered sequence $T_1,\dots,T_m$. With the $\bde\in \NN^m$, $\bdX^{\bde}$ will stand for the product $X_1^{e_1}X_2^{e_2}\cdots X_m^{e_m}$ in the suitable ring. We can similarly define $\bdT^{\bde}$. 

To better express the properties of the ideal $L$, we will also use an accompanied $R$-algebra homomorphism $\Psi: S\to R$ which is determined by 
\[
X_{i}\mapsto T_{i}^{a_i} \quad \text{for $1\le i\le m$, \qquad and}\quad W\mapsto \bdT^{\bdb}:=T_1^{b_1}T_2^{b_2}\cdots T_m^{b_m}.
\] 
Meanwhile, we will write $X_{m+1}$ for $W$, and $\widetilde{\bdX}$ for the ordered sequence $X_1,\dots,X_m,W$. With the vector $\bdbeta\in \NN^{m+1}$, the notation of $\widetilde{\bdX}^{\bdbeta}$ will be understood likewise as above.

It is quite clear that the defining ideal $L$, being the kernel of the graded homomorphism $\Phi$, is graded with respect to $\widetilde{\bdX}$.  For this reason, we will write $L=\bigoplus_{i\ge 1}L_i$, with $L_i$ being the degree $i$ piece of $L$.
It is also well-known that $L$ is generated by the binomials of the form $\bdT^{\bdalpha}\widetilde{\bdX}^{\bdbeta}-\bdT^{\bdgamma}\widetilde{\bdX}^{\bddelta}$, with
$\bdalpha,\bdgamma\in \NN^m$ and $\bdbeta,\bddelta\in \NN^{m+1}$ such that
$\left| \bdbeta \right|=\left| \bddelta \right|$ and $\gcd(\bdT^{\bdalpha},\bdT^{\bdgamma})=1$ as well as $\gcd(\widetilde{\bdX}^{\bdbeta},\widetilde{\bdX}^{\bddelta})=1$; see, for instance, \cite[Corollary 4.3]{MR1363949} or \cite{MR2611561}. Indeed, the binomial $\bdT^{\bdalpha}\widetilde{\bdX}^{\bdbeta}-\bdT^{\bdgamma}\widetilde{\bdX}^{\bddelta}$ here is uniquely determined by the pair $(\widetilde{\bdX}^{\bdbeta},\widetilde{\bdX}^{\bddelta})$, since the monomials
\[
\bdT^{\bdalpha}=\frac{\Psi(\widetilde{\bdX}^{\bddelta})}{\gcd(\Psi(\widetilde{\bdX}^{\bdbeta}),\Psi(\widetilde{\bdX}^{\bddelta}))} \qquad \text{and}\qquad
\bdT^{\bdgamma}=\frac{\Psi(\widetilde{\bdX}^{\bdbeta})}{\gcd(\Psi(\widetilde{\bdX}^{\bdbeta}),\Psi(\widetilde{\bdX}^{\bddelta}))}.
\]
For later reference, we will write this binomial as $\calP(\widetilde{\bdX}^{\bdbeta},\widetilde{\bdX}^{\bddelta})$. Thus, the defining ideal $L$ is generated by the infinite set
\[
\Gamma_{-1}:=\Set{\calP(\widetilde{\bdX}^{\bdbeta},\widetilde{\bdX}^{\bddelta})|\left| \bdbeta \right|=\left| \bddelta \right|\text{ and }\gcd(\widetilde{\bdX}^{\bdbeta},\widetilde{\bdX}^{\bddelta})=1}.
\]
For the sake of an efficient investigation, we can further impose a special format on these binomials, due to the nature of the given monomial generators in $I$.

First of all, notice that the linear piece $L_1$ of the defining ideal $L$ with respect to $\widetilde{\bdX}$ is generated by 
\begin{equation}
\Gamma_0:=\Set{\calP(X_i,X_j)|1\le j<i\le m+1}. \label{eqn:linear-parts-m}
\end{equation} 
They come from the Koszul syzygies of the generators of $I$. Using the language of \cite{MR3490460}, the generators in $\Gamma_0$ with $i=m+1$ are called \emph{reduced relations} while the remaining ones are called \emph{Koszul relations}.  

Meanwhile, a Gr\"obner-basis argument will be involved later. Thus, we need to equip the polynomial ring $S=\KK[T_1,\dots,T_m,X_1,\dots,X_m,W]$ with a monomial ordering.  In the following, we will denote by $\tau$ the lexicographic order on $S$ with respect to $W>X_m>X_{m-1}>\cdots>X_1>T_1>\cdots > T_m$.  If $f$ is a polynomial in $S$, we will simply use $\ini(f)$ to denote the leading monomial of $f$ with respect to $\tau$. Furthermore, if $J$ is an ideal of $S$, then we will use $\ini(J)$ to denote the initial ideal of $J$ with respect to $\tau$.

Now, we are ready to give the first description of the defining ideal $L$.

\begin{Proposition}
	\label{prop:more-generators-m}
	The binomials in $\Gamma_0$ as well as those in
	\begin{equation}
	\Gamma_1:=\left\{\calP(W^{|\bdc|},\bdX^{\bdc})\,\Big|\, \bdzero \ne \bdc \in \NN^m\right\}
	\end{equation}
	provide an infinite Gr\"obner basis of the defining ideal $L$.
\end{Proposition}

In the following, we will use repeatedly the easy fact: if $A\widetilde{\bdX}^{\bdbeta}-B\widetilde{\bdX}^{\bddelta}\in L$ with nonzero monomials $A,B\in R$, then this binomial is a multiple of $\calP(\widetilde{\bdX}^{\bdbeta},\widetilde{\bdX}^{\bddelta})$.

\begin{proof}
	[{Proof of Proposition \ref{prop:more-generators-m}}]
	Notice that one can find a Gr\"obner basis of $L$ in $\Gamma_{-1}$ and $\ini(L)=\braket{\ini(g)\mid g\in \Gamma_{-1}}$. Thus, it suffices to take an arbitrary nonzero $g:=\calP(\widetilde{\bdX}^{\bdbeta},\widetilde{\bdX}^{\bddelta})$ in $\Gamma_{-1}$, and show that 
	\begin{itemize}
		\item $g\in \braket{\Gamma_0\cup \Gamma_1}$, as well as
		\item $\ini(g)$ is divisible by the initial monomial of some binomial in $\Gamma_0\cup \Gamma_1$.
	\end{itemize}
	We have two cases.
	\begin{enumerate}[a]
		\item \label{prop:more-generators-m-a}
		If $W\nmid \widetilde{\bdX}^{\bdbeta}$ as well as $W\nmid \widetilde{\bdX}^{\bddelta}$, then $g$ belongs to the defining ideal $\widetilde{L}$ of the Rees algebra associated to the complete intersection ideal $\widetilde{I}:=\braket{T_1^{a_1},T_2^{a_2},\dots,T_{m}^{a_m}}$.  To study its presentation, we will use simply $\Phi|_{R[X_1,\dots,X_m]}$.  Since $\widetilde{I}$ is generated by a $d$-sequence, it follows from \cite{MR563225} that 
		\begin{equation}
		\widetilde{L}= \braket{\calP(X_i,X_j)\mid 1\le j<i\le  m} \subset R[X_1,X_2,\dots,X_{m}]. 
		\label{eqn:tildeL-m}
		\end{equation}
		In particular, $g\in \braket{\Gamma_0}\subset \braket{\Gamma_0\cup \Gamma_1}$. 
		
		Without loss of generality, for $g$, we may assume that $X_{i}\mid \widetilde{\bdX}^{\bdbeta}$ and $X_{j}\mid \widetilde{\bdX}^{\bddelta}$  for some $i,j$ with $1\le j<i\le m$.  Now, one can verify directly that $\ini(\calP(X_i,X_j))\,|\,\ini(g)$.
		
		\item \label{prop:more-generators-m-b}
		Therefore, in the following, we may assume that $W \,|\, \widetilde{\bdX}^{\bdbeta}$ by symmetry. To exclude the trivial case, we may require in the following that $g\notin \Gamma_1$. Thus, we can find that some $i_0,j_0$ with $1\le i_0\ne j_0\le m$ such that $X_{i_0}\mid \widetilde{\bdX}^{\bdbeta}$ and $X_{j_0}\mid \widetilde{\bdX}^{\bddelta}$.  For simplicity, we write $g=A\widetilde{\bdX}^{\bdbeta}-B\widetilde{\bdX}^{\bddelta}$. As $\gcd(\widetilde{\bdX}^{\bdbeta},\widetilde{\bdX}^{\bddelta})=1$, by a direct computation, we see that $T_{i_0}^{a_{i_0}\beta_{i_{0}}+b_{i_0}\beta_{{m+1}}}$ is a factor of $B$. As $\beta_{i_0}\ge 1$, the exponent of this factor is at least $a_{i_0}$. On the other hand, we have $\calP(X_{i_0},X_{j_0})=T_{j_0}^{a_{j_0}}X_{i_0}-T_{i_0}^{a_{i_0}}X_{j_0}\in \Gamma_0$. Thus,
		\begin{align*}
		g&= \frac{B
			\widetilde{\bdX}^{\bddelta}}{T_{i_0}^{a_{i_0}}X_{j_0}}\calP(X_{i_0},X_{j_0})+X_{i_0}\left(
		\frac{A\widetilde{\bdX}^{\bdbeta}}{X_{i_0}}-\frac{BT_{j_0}^{a_{j_0}}\widetilde{\bdX}^{\bddelta}}{T_{i_0}^{a_{i_0}}X_{j_0}} \right).
		\end{align*}
		As $g$ and $\calP(X_{i_0},X_{j_0})$ belong to the prime defining ideal $L$, so does
		\[
		\frac{A\widetilde{\bdX}^{\bdbeta}}{X_{i_0}}-\frac{BT_{j_0}^{a_{j_0}}\widetilde{\bdX}^{\bddelta}}{T_{i_0}^{a_{i_0}}X_{j_0}}.
		\]
		This implies that the above binomial is a multiple of $\calP(\widetilde{\bdX}^{\bdbeta}/{X_{i_0}},\widetilde{\bdX}^{\bddelta}/X_{j_0})$.  But by induction on $\left| \bdbeta \right|=\left| \bddelta \right|$, we see that $\calP(\widetilde{\bdX}^{\bdbeta}/{X_{i_0}},\widetilde{\bdX}^{\bddelta}/X_{j_0})\in \braket{\Gamma_0\cup \Gamma_1}$. Thus, $g\in \braket{\Gamma_0\cup \Gamma_1}$. 
		
		Meanwhile, we observe that $W$ is not involved in $\frac{B \widetilde{\bdX}^{\bddelta}}{T_{i_0}^{a_{i_0}}X_{j_0}}\calP(X_{i_0},X_{j_0})$. Thus, 
		\[
		\ini(g)=X_{i_0}\cdot \ini \left( \frac{A\widetilde{\bdX}^{\bdbeta}}{X_{i_0}}-\frac{BT_{j_0}^{a_{j_0}}\widetilde{\bdX}^{\bddelta}}{T_{i_0}^{a_{i_0}}X_{j_0}} \right),
		\]
		which in turn will be a multiple of $\ini(\calP(\widetilde{\bdX}^{\bdbeta}/{X_{i_0}},\widetilde{\bdX}^{\bddelta}/X_{j_0})$).  But again by induction, the latter is divisible by the leading monomial of some binomial in $\Gamma_0\cup \Gamma_1$. Thus, this completes our proof.
				\qedhere
	\end{enumerate}
\end{proof}

Since $\Gamma_0\cup \Gamma_1$ is an infinite Gr\"obner basis of the defining ideal $L$, we can find a finite subset $\Gamma_2$ which is still a Gr\"obner basis of $L$. From $\Gamma_2$, we can compute the reduced Gr\"obner basis $\Gamma$ of $L$. Whence, the leading terms of the binomials in $\Gamma$ are mutually non-comparable, and the tails of each binomial in $\Gamma$ cannot be reduced further by $\Gamma$. It is well-known that $\Gamma$ is uniquely determined by $\tau$. Furthermore, it is not difficult to see that $\Gamma_0\subseteq \Gamma\subseteq \Gamma_0\cup \Gamma_1$ in our situation. 

In \cite{MR3490460}, \cite{MR3320795} and \cite{MR3295062}, the reduced Gr\"obner basis of the defining ideal $L$ is explicitly constructed via moving curve or Sylvester forms techniques. And this Gr\"obner basis is shown to be a minimal generating set of $L$ as well in each case. Thus, we also want to determine when our $\Gamma$ is also a minimal generating set of $L$, although we have not written down $\Gamma$ explicitly. Without loss of generality, we may assume that $b_1\le b_2\le \cdots \le b_m$. Although we have an implicit order over the $T$'s and $X$'s with respect to $\tau$, it can be seen from the following proof that the order over the $b$'s here is not essential. 
Notice that by our assumptions in the first section, we will in turn have $b_{m-1},b_m>0$.

\begin{Proposition}
	\label{prop:min-basis}
	With the assumptions above, the reduced Gr\"obner basis $\Gamma$ is a minimal generating set of $L$ if and only if $b_{m-2}>0$.
\end{Proposition}

\begin{proof}
    We have two cases.
	\begin{enumerate}[a]
		\item If $b_1=b_2=\cdots=b_{m-2}=0$, we claim that $\Gamma$ is not a minimal generating set of $L$. To see this, consider the following three distinct binomials in $\Gamma_0$:
		\begin{align*}
		\calP(X_{m},X_{m-1})&=T_{m-1}^{a_{m-1}}X_m-T_{m}^{a_m}X_{m-1},\\
		\calP(W,X_{m-1})&=T_{m-1}^{a_{m-1}-b_{m-1}}W-T_{m}^{b_m}X_{m-1},\\
		\calP(W,X_m)&=T_{m}^{a_{m}-b_m}W-T_{m-1}^{b_{m-1}}X_m.
		\end{align*}
		Since
		\[
		\calP(X_m,X_{m-1})=T_m^{a_{m}-b_m}\calP(W,X_{m-1})-T_{m-1}^{a_{m-1}-b_{m-1}}\calP(W,X_m),
		\]
		$\Gamma$ is not a minimal generating set.
		\item  In the remaining case, we will have $b_{m-2}>0$. Whence, at least two distinct $T$-variables will show up in each tail monomial of the reduced relations.
		
		In this situation, we claim that $\Gamma$ is a minimal generating set of $L$.  Suppose for contradiction that this is not true. Then, there exists some $F\in \Gamma$ such that $F=\sum_{G\in \Gamma\setminus\{F\}} k_G G$ with $k_G\in S$.   
		
		If $F$ is a reduced relation, then $F\in \Gamma_0\cap \Gamma_1$. Whence, we can check with ease that none of the monomials emerging in $\Gamma \setminus \{F\}$ is a factor of $\ini(F)$. Indeed, as $\ini(F)$ takes the form $AW^{k}$ for some monomial $A\in R$ and some positive integer $k$, we only need to check the leading monomial of each $G$ here. But this is clear since $\Gamma$ is a reduced Gr\"obner basis. Thus, the equality $F=\sum_{G\in \Gamma\setminus\{F\}} k_G G$ cannot happen in this subcase.  
		
		Otherwise, $F$ will be a Koszul relation. Say, $F=\calP(X_{i},X_j)$ with $1\le j<i\le m$. We can check with ease that none of the monomials emerging in $(\Gamma_0\cup \Gamma_1)\setminus \{F\}$ is a factor of $\ini(F)=T_j^{a_j}X_i$. Indeed, since $F$ is linear with respect to $\widetilde{\bdX}$, it suffices to investigate the reduced relations. As this can be verified directly, the equality $F=\sum_{G\in \Gamma\setminus\{F\}} k_G G$ still cannot happen in this subcase. 
				\qedhere
	\end{enumerate}
\end{proof}

In Section 4, we will talk about the maximal degree of the Gr\"obner basis of $L$ when $I$ is equi-generated.

\section{The almost Cohen--Macaulay conjecture of Vasconcelos}
In this section, we will confirm the conjecture of Vasconcelos (\cite[Conjecture 4.15]{MR3096902}) which states that the Rees algebra of an Artinian almost complete intersection monomial ideal $I$ is \emph{almost Cohen--Macaulay}, i.e., $\depth_S (\calR(I))\ge \dim(\calR(I))-1$. 

To achieve this goal, in the following, we give a total order $\tau'$ to $\NN^m$ as follows: for $\bdalpha,\bdbeta\in \NN^m$, $\bdalpha\prec_{\tau'}\bdbeta$ if and only if $|\bdalpha|<|\bdbeta|$, or 
\[
\text{$|\bdalpha|=|\bdbeta|$ and 
	$\alpha_{i_0}>\beta_{i_0}$ for some $i_0$ while $\alpha_i=\beta_i$ for all $i>i_0$.}
\]
Thus, with respect to  $\tau'$, the initial vector is $\bdzero=(0,0,\dots,0)$.

Considering Proposition \ref{prop:more-generators-m}, we can always find a finite subset $\calC\subset \{\bdc\in\NN^m\mid |\bdc|\ge 2\}$ and $\Gamma_2:=\{\calP(W^{|\bdc|},\bdX^{\bdc})\mid \bdc\in \calC\}$ such that the binomials in the disjoint union $\Gamma_0\sqcup \Gamma_2$ provide a Gr\"obner basis of the defining ideal $L$.  For our induction proof below, we will harmlessly assume that $\calC$ is \emph{closed} in the sense that if $\bdc'\prec_{\tau'}\bdc$ with $\bdc\in \calC$ and $|\bdc'|\ge 2$, then $\bdc'\in \calC$.

We can order the vectors in $\calC$ as $\bdc_1\prec_{\tau'}\cdots\prec_{\tau'}\bdc_N$. For
$1\le i\le N$, we will write $\bdc_i=(c_{i,1},c_{i,2},\dots,c_{i,m})\in \NN^m$.  Furthermore, for $1\le j\le N$, we set
\[
\Theta_j:=\Gamma_0\cup \{\calP(W^{|\bdc_i|},\bdX^{\bdc_i})\mid 1\le i\le j\} \quad\text{ and }\quad \calH_j:=\braket{\Theta_j}.
\]
By convention, $\Theta_0=\Gamma_0$ and correspondingly $\calH_0=L_1S$.

\begin{Proposition}
	\label{prop:GB-subideal}
	With the notation above, for $0\le j\le N$, $\Theta_j$ is a Gr\"obner  basis of the ideal $\calH_j$.
\end{Proposition}

\begin{proof}
	We will apply the standard technique of checking S-pairs $\spoly(F,G)$ for $F\ne G\in \Theta_j$; cf.~\cite[Section 2.3]{MR2724673}. By symmetry, we have the following three cases.
	\begin{enumerate}[a]
		\item \label{prop:GB-subideal-a}
		Assume that both $F$ and $G$ are Koszul relations. It is clear that $\spoly(F,G)\in \widetilde{L}$ in the equation \eqref{eqn:tildeL-m}. Thus, $\spoly(F,G)$ can be reduced to zero by the multivariate division algorithm relative to the set $\Gamma_0$; cf.~part \ref{prop:more-generators-m-a} of the proof of Proposition \ref{prop:more-generators-m}.		
		\item Assume that $F$ is a Koszul relation while $G$ is not. Thus, we may write $G=\calP(W^{|\bdbeta|},\bdX^{\bdbeta})$ and $F=\calP(X_{j_0},X_{i_0})$ with $1\le i_0<j_0\le m$. Whence, $\spoly(G,F)$ is the binomial $AW^{|\bdbeta|}X_{i_0}-BX_{j_0}\bdX^{\bdbeta}$ for some monomial $A,B\in R$. We have two subcases.
		
		\begin{enumerate}[i]
			\item Suppose first that $X_{i_0}$ is not a factor of $\bdX^{\bdbeta}$.
			Since obviously $X_{i_0}\mid W^{|\bdbeta|}X_{i_0}$ and $X_{j_0}\mid X_{j_0}\bdX^{\bdbeta}$, we can apply the argument as in the case \ref{prop:more-generators-m-b} in the proof of Proposition \ref{prop:more-generators-m}. Namely, $T_{i_0}^{a_{i_0}}$ is a factor of $B$ and we may express $AW^{|\bdbeta|}X_{i_0}-BX_{j_0}\bdX^{\bdbeta}$ as a friendly $S$-linear combination of $F$ and $G$:
			\begin{align*}
			AW^{|\bdbeta|}X_{i_0}-BX_{j_0}\bdX^{\bdbeta}&= \frac{B{\bdX}^{\bdbeta}}{T_{i_0}^{a_{i_0}}}   \calP(X_{i_0},X_{j_0})+X_{i_0}\left(AW^{|\bdbeta|}-\frac{BT_{j_0}^{a_{j_0}}{\bdX}^{\bdbeta}}{T_{i_0}^{a_{i_0}}}\right)\\
			&= \frac{B{\bdX}^{\bdbeta}}{T_{i_0}^{a_{i_0}}}F+DX_{i_0}G
			\end{align*}
			for suitable monomial $D\in R$. Because of the existence of $W^{|\bdbeta|}$ in $\ini(G)$,
			\[
			\ini(AW^{|\bdbeta|}X_{i_0}-BX_{j_0}\bdX^{\bdbeta})=DX_{i_0}\ini(G).
			\]
			Thus, $AW^{|\bdbeta|}X_{i_0}-BX_{j_0}\bdX^{\bdbeta}$ can be reduced by the multivariate division algorithm via $F$ and $G$ to zero. Therefore, we are done in this subcase.
			
			\item Suppose that $X_{i_0}$ is a factor of $\bdX^{\bdbeta}$. We may instead consider the binomial $AW^{|\bdbeta|}-BX_{j_0}\bdX^{\bdbeta}/X_{i_0}$. Suppose that $\bdX^{\bdbeta'}=X_{j_0}\bdX^{\bdbeta}/X_{i_0}$. Then $AW^{|\bdbeta|}-BX_{j_0}\bdX^{\bdbeta}/X_{i_0}$ is simply a multiple of $G'=\calP(W^{|\bdbeta'|},\bdX^{\bdbeta'})$. As $|\bdbeta|=|\bdbeta'|$ while $\bdbeta'\prec_{\tau'} \bdbeta$, $G'\in \Theta_j$. Therefore, we are done in this subcase.
		\end{enumerate}
		
		\item Assume that neither $F$ nor $G$ is a Koszul relation. Thus, we will write $F=\calP(W^{|\bdbeta|},\bdX^{\bdbeta})$ and $G=\calP(W^{|\bddelta|},\bdX^{\bddelta})$ for some $\bdbeta\precneq \bddelta\preceq \bdc_j$ with respect to  $\tau'$. If $|\bdbeta|=|\bddelta|$, then $\spoly(F,G)\in \widetilde{L}$. Thus, as in the case \ref{prop:GB-subideal-a}, we are done. Therefore, we may assume that $|\bdbeta|<|\bddelta|$. Whence, $\spoly(F,G)$ is a multiple of $\calP(W^{|\bddelta|-|\bdbeta|}\bdX^{\bdbeta'},\bdX^{\bddelta'})$ for
		\[
		\bdX^{\bdbeta'}= \bdX^{\bdbeta}/\gcd(\bdX^{\bdbeta},\bdX^{\bddelta}) \quad \text{and} \quad
		\bdX^{\bddelta'}= \bdX^{\bddelta}/\gcd(\bdX^{\bdbeta},\bdX^{\bddelta}).
		\]
		Being in the defining ideal $L$, the binomial $\calP(W^{|\bddelta|-|\bdbeta|}\bdX^{\bdbeta'},\bdX^{\bddelta'})$ can be reduced to zero by the multivariate division algorithm relative to the set $\Gamma_0\cup \Gamma_1$; cf.~Proposition \ref{prop:more-generators-m}. Notice that our $\calC$ is closed. By checking the degree with respect to $W$, we can expect that each emerging binomial is either in $\Gamma_0$ or is $\calP(W^{|\bdc_i|},\bdX^{\bdc_i})$ for some $\bdc_i$ with $|\bdc_i|\le |\bddelta|-|\bdbeta|$. In particular, these $i$'s will be strictly less than the given $j$. Thus, we are done in this last case.  
				\qedhere
	\end{enumerate}
\end{proof}

Recall that an ideal $J$ of $S$ is called \emph{extended} with respect to the inclusion $R\subset S$ if $J=(J\cap R)S$.

\begin{Proposition}
	\label{prop:colon-ideals-are-extended}
	With the notation above, for $1\le j\le N$, the colon ideal
	\begin{equation}
	\ini (\calH_{j-1}):\ini(\calP(W^{|\bdc_j|},\bdX^{\bdc_j}))
	\label{colon-ideal}
	\end{equation}
	can be extended from $R$.
\end{Proposition}

\begin{proof}
	By Proposition \ref{prop:GB-subideal}, it suffices to consider the following three types of simpler colon ideals.
	\begin{enumerate}[a]
		\item For $1\le i\le j-1$, consider the colon ideal 
		\begin{equation}
		\braket{\ini(\calP(W^{|\bdc_i|},\bdX^{\bdc_i}))}:
		\braket{\ini(\calP(W^{|\bdc_j|},\bdX^{\bdc_j}))}.
		\label{colon-1}
		\end{equation}
		Notice that the emerging two leading monomials take the form	$AW^{|\bdc_i|}$ and $A'W^{|\bdc_j|}$ respectively with $A,A'\in R$. As $|\bdc_i|\le |\bdc_j|$, the colon ideal in \eqref{colon-1} is extended.
		
    \item For $1\le i\le m$, consider the colon ideal
        \begin{equation}
            \braket{\ini(\calP(W,X_i))}:
            \braket{\ini(\calP(W^{|\bdc_j|},\bdX^{\bdc_j}))}.
            \label{colon-2}
        \end{equation}
        Notice that the emerging two leading monomials take the form	$AW$ and $A'W^{|\bdc_j|}$ respectively with $A,A'\in R$. As $1< |\bdc_j|$, the colon ideal in \eqref{colon-2} is extended.
    \item For $1\le i_1<i_2\le m$, consider the colon ideal 
        \begin{equation}
            \braket{\ini(\calP(X_{i_2},X_{i_1}))}:\braket{\ini(\calP(W^{|\bdc_j|},\bdX^{\bdc_j}))}.
            \label{colon-3}
        \end{equation}
        Since $W\mid \ini(\calP(W^{|\bdc_j|},\bdX^{\bdc_j}))$ and $W\nmid \ini(\calP(X_{i_2},X_{i_1}))$, we have
        \[
            \braket{\ini(\calP(X_{i_2},X_{i_1}))}: \braket{\ini(\calP(W^{|\bdc_j|},\bdX^{\bdc_j}))}
            \subseteq
            \braket{\ini(\calP(X_{i_2},X_{i_1}))W}: \braket{\ini(\calP(W^{|\bdc_j|},\bdX^{\bdc_j}))}.
        \]
        Then, since $\ini(\calP(X_{i_2},X_{i_1}))W=T_{i_1}^{a_{i_1}}X_{i_2}W\in \braket{\ini(\calP(W,X_{i_1}))}$, the colon ideal in \eqref{colon-3} is a subideal of the colon ideal in \eqref{colon-2} with $i=i_1$.  Thus, this case is harmless.
\end{enumerate}
From what we have found above, we can conclude that colon ideal in \eqref{colon-ideal} equals
\[
    \sum_{i=1}^{j-1} \braket{\ini(\calP(W^{|\bdc_i|},\bdX^{\bdc_i}))}:
    \braket{\ini(\calP(W^{|\bdc_j|},\bdX^{\bdc_j}))}
    +
    \sum_{i=1}^{m} \braket{\ini(\calP(W,X_i))}:
    \braket{\ini(\calP(W^{|\bdc_j|},\bdX^{\bdc_j}))},
\]
which is extended with respect to the inclusion $R\subset S$.
\end{proof}

Here is the third piece of our attack of the Vasconcelos' conjecture. 
This piece is one of the key ingredients in the proofs of \cite[Theorem 2.5]{MR3490460} and \cite[Theorem 3.14]{MR3295062}. 

\begin{Lemma}
    \label{rmk:key-points}
    Let $\calH$ be an ideal generated by elements $f_1,\dots,f_r$ in $S=R[X_1,\dots,X_m,W]$ for $m=\dim(R)$. If the colon ideal $\braket{f_1,\dots,f_{r-1}}:f_r$ is extended with respect to $R\subset S$ and 
    \[
        \depth_S(S/\braket{f_1,\dots,f_{r-1}})\ge m,
    \]
    then $\depth_S(S/\calH)\ge m$.  
\end{Lemma}

\begin{proof}
    Notice that $\dim(S)=2m+1$. If the colon ideal is extended, then 
    \[
        \depth_S (S/\braket{f_1,\dots,f_{r-1}}:f_r)\ge m+1
    \]
    by the flat-base-change technique.  After that, we can apply the well-known Depth Lemma to the short exact sequence
    \begin{equation}
        0\to S/\braket{f_1,\dots,f_{r-1}}:f_r \xrightarrow{\cdot f_r} S/\braket{f_1,\dots,f_{r-1}} \to S/\calH \to 0
        \label{eqn:SES}
    \end{equation}
    to get the desired estimate.
\end{proof}

Now, it is time to state the last main piece of our intended proof.  It will serve as the base case of the induction argument.

\begin{Proposition}
    \label{prop:InitSymmAlgIdealIsCM}
    With the notation as before, the initial ideal $\ini(\calH_0)$ in $S=R[X_1,\dots,X_m,W]$ satisfies $\depth(S/\ini(\calH_0))\ge m+1$.
\end{Proposition}

\begin{proof}
    We have seen in Proposition \ref{prop:GB-subideal} that $\Gamma_0$ provides a Gr\"obner basis of $\calH_0$. Thus, $\ini(\calH_0)=\braket{\ini(\calP(X_i,X_j))\mid 1\le j<i\le m+1}$. Regarding these generator, indexed by the tuple $(i,j)$, we implicit an ordering $\prec_s$ for the following induction argument:  
    \begin{align*}
        (2,1)
        \prec_s (3,1)
        \prec_s \cdots
        \prec_s (m,1)
        \prec_s (3,2)
        \prec_s (4,2)
        \prec_s \cdots
        \prec_s (m,2)
        \prec_s \cdots 
        \prec_s (m-1,m-2)\\
        \prec_s (m,m-2)
        \prec_s (m,m-1)
        \prec_s (m+1,1)
        \prec_s (m+1,2)
        \prec_s \cdots
        \prec_s (m+1,m)
    \end{align*}
    For each $(i_0,j_0)$ with $1\le j_0<i_0\le m+1$, define
    \[
        K_{i_0,j_0}\coloneqq \braket{\ini(\calP(X_i,X_j))\mid (i,j)\preceq_s (i_0,j_0)} \subset S
    \]
    and
    \[
        K_{i_0,j_0}'\coloneqq \braket{\ini(\calP(X_i,X_j))\mid (i,j)\prec_s (i_0,j_0)} \subset S. 
    \] 
    In the remaining part of this proof, we will use mathematical induction with respect to the ordering of $\prec_s$ and show
    \[
        \depth(S/K_{i,j})\ge
        \begin{cases}
            m+2,& \text{if } (i,j)\prec_s (m+1,1), \\
            m+1, & \text{otherwise}.
        \end{cases}
    \]
    Once the above estimate is established, as $\ini(\calH_0)=K_{m+1,m}$, we will have $\depth(S/\ini(\calH_0))\ge m+1$, as expected.
    \begin{enumerate}[a]
        \item \label{case-a}
            Firstly, consider the cases when $(i,j)\prec_s (m+1,1)$.
            Notice that $\ini(\calP(X_i,X_j))=T_j^{a_j}X_i$ in this situation.
            Hence, for the base case of the discussion, we easily get 
            \[
                \depth(S/K_{2,1})=\dim(S)-1=2m\ge m+2.
            \]
            As for the general $(i,j)$ in this case, one can verify directly with ease that 
            \[
                K_{i,j}':\ini(\calP(X_i,X_j))=\braket{T_1^{a_1},T_2^{a_2},\dots,T_{j-1}^{a_{j-1}},X_{j+1},X_{j+2},\dots,X_{i-1}}.
            \]
            Therefore, 
            \[
                \depth(S/K_{i,j}':\ini(\calP(X_i,X_j)))\ge 2m+1-(j-1+(i-1-j))\ge m+3.
            \]
            By applying the Depth Lemma to sequences like that in \eqref{eqn:SES}, we can use induction to get 
            \[
                \depth(S/K_{i,j})\ge m+2
            \]
            in this case.

        \item In the remaining cases, we have $(i,j)=(m+1,j)$. 
            First of all, we notice that $K_{m+1,1}'=K_{m,m-1}$. And by the result in case \ref{case-a}, we already have
            \[
                \depth(S/K_{m,m-1})\ge m+2.
            \]
            Meanwhile, $\ini(\calP(X_i,X_j))=T_j^{a_j-b_j}W$ for general $(i,j)$ in this case. And one can verify directly with ease that $K_{i,j}':\ini(\calP(X_i,X_j))$ is generated by
            \[
                T_1^{a_1-b_1},T_2^{a_2-b_2},\dots,T_{j-1}^{a_{j-1}-b_{j-1}}
            \]
            together with
            \begin{align}
                T_j^{b_j}X_{j+1},
                T_j^{b_j}X_{j+2},
                \dots
                T_j^{b_j}X_{m},
                T_{j+1}^{a_{j+1}}X_{j+2},
                T_{j+1}^{a_{j+1}}X_{j+3},
                \dots, 
                T_{j+1}^{a_{j+1}}X_{m}, \notag\\
                \dots,
                T_{m-2}^{a_{m-2}}X_{m-1},
                T_{m-2}^{a_{m-2}}X_{m},
                T_{m-1}^{a_{m-1}}X_{m}. \label{eqn:group2}
            \end{align}
            The ideal generated by the second group \eqref{eqn:group2} is essentially $K_{m,m-1}$, however, with the parameters $m$ replaced by $m-(j-1)$ and $(a_1,\dots,a_m)$ replaced by $(b_j,a_{j+1},\dots,a_m)$. Thus, by the result in the previous case \ref{case-a},
            \begin{align*}
                & \depth(S/(K_{i,j}':\ini(\calP(X_i,X_j))))\ge \\
                & \qquad  \qquad \qquad (2(j-1)-(j-1))+(m-(j-1)+2)=m+2.
            \end{align*}
            Now, we can apply the Depth Lemma again to the short exact sequences
            \[
                0\to S/(K_{i,j}':\ini(\calP(X_i,X_j))) \to S/K_{i,j}' \to S/K_{i,j}\to 0
            \]
            for $1\le j\le m$ and use induction to have
            \[
                \depth(S/K_{i,j})\ge m+1,
            \]
            as expected. 
            \qedhere
    \end{enumerate}
\end{proof}

\begin{Remark}
    In our circumstance, 
    \[
        S/L_1S=S/\calH_0\cong \Sym(I),
    \]
    the symmetric algebra of $I$.
    Since $I$ is an Artinian almost complete intersection ideal, $\Sym(I)$ is Cohen--Macaulay by \cite[Corollary 10.2]{MR686942}. In particular, this algebra is unmixed having dimension $\dim(R)+1=m+1$ by \cite[Proposition 8.5]{MR686942}.
    Notice that 
    \[
        \dim (S/\calH_0)=\dim(S/\ini(\calH_0)) \quad \text{and} \quad \depth (S/\calH_0)\ge \depth(S/\ini(\calH_0));
    \]
    cf.~\cite[Theorem 3.3.4]{MR2724673} and \cite[Theorem 1.3]{MR2153889}.
    Thus, $S/\ini(\calH_0)$ is also Cohen--Macaulay. 
\end{Remark}

Finally, we can prove our main result of this paper.

\begin{Theorem}
    \label{thm:ACM}
    With the given monomial Artinian ideal $I$ given in the equation \eqref{eqn:I}, the Rees algebra $\calR(I)$ is almost Cohen--Macaulay.
\end{Theorem}

\begin{proof}
    Notice that $\calR(I)\cong S/L$ while it is well-known that
    \[
        m+1=\dim (S/L)=\dim(S/\ini(L)) \quad \text{and} \quad
        \depth (S/L)\ge \depth(S/\ini(L)).
    \]
    Thus, it suffices to show that $S/\ini(L)$ is almost Cohen--Macaulay, i.e., to show that 
    \[
        \depth(S/\ini(L))\ge m.
    \] 
    Notice that $\ini(L)=\calH_N$. We will apply the technique outlined in Lemma \ref{rmk:key-points} repeatedly towards the ideals $\calH_j$, $1\le j\le N$. Since we have shown in Proposition \ref{prop:colon-ideals-are-extended} that the successive colon ideals are extended, it remains to verify that $\depth(S/\ini(\calH_0)\ge m$. And this is covered (with a stronger result) in Proposition \ref{prop:InitSymmAlgIdealIsCM}.
\end{proof}

In particular, we are able to cover the corresponding almost Cohen--Macaulay results in  \cite[Theorem 2.5]{MR3490460}, \cite[Proposition 4.12]{MR3096902} and \cite[Theorems 3.7 and 3.14]{MR3295062}.

\section{Applications to the equi-generated case}

In this last section, we will focus on the equi-generated case when $a_1=a_2=\cdots=a_m=|\bdb|$, except for Theorem \ref{prop:smaller-generators-m} which holds for slightly more general situation.  

Notice that finding the minimal generators of the defining ideal $L$ in the equi-generated case is also strongly related to the problem of computing the moving curve ideals of monomial rational parameterizations of the form
\[
    \varphi: \PP_{\KK}^m \to \PP_{\KK}^{m+1}, \qquad (t_1:t_2:\cdots:t_m)\mapsto (t_1^{\left| \bdb \right|}:t_2^{|\bdb|}:\cdots:t_m^{|\bdb|}:t_1^{b_1}t_2^{b_2}\cdots t_m^{b_m}).
\]
See for instance \cite{MR3320795} along this line. However,  from the work of \cite{MR3320795}, \cite{MR2013172} and \cite{MR3295062}, one can see that finding the explicit presentation of the minimal generating set is very difficult. One must impose very strong conditions in order to write out the minimal generating set explicitly. The next question will be if one can at least find a bound on the degree of generators of the defining ideal $L$ without narrowing down to special cases. It is unfortunate that even for a degree $3$ square-free monomial ideal, one can get a generator of any degree if no other restriction is placed; see for example \cite{MR3853064} and \cite{arXiv:1809.08351}. Finding the degree bound is actually a classical question in the  commutative algebra community.  Recall that with the $\ZZ$-graded ideal $L=\bigoplus_{i\ge 1}L_i$, the \emph{relation type} of the Rees algebra $\calR(I)$ is defined to be
\[
    \reltype(I):=\inf\Set{s|L=\braket{L_1,L_2,\dots,L_s}}.
\]
This invariant, which is independent of the chosen set of generators of $I$, gives us some idea on how far the Rees algebra $\calR(I)$ is from the symmetric algebra $\Sym(I)$ and provides us an insight of the complexity of the Rees algebra.

In the following, we want to fine-tune the result in Proposition \ref{prop:more-generators-m} when $|\bdb|\le \min(\bda):=\min(a_1,\dots,a_m)$. 

\begin{Theorem}
    \label{prop:smaller-generators-m}
    With assumptions as before, suppose further that $|\bdb|\le \min(\bda)$. Then, the binomials in $\Gamma_0$ as well as those in
    \begin{equation}
        \Gamma_3:= \left\{\calP(W^{|\bdc|},\bdX^{\bdc})\,\Big|\, \bdzero\ne \bdc\le_{\tau''} \bdb\right\}.
        \label{eqn:higher-parts-m}
    \end{equation}
    provide a \textup{(}non-reduced\textup{)} Gr\"obner basis of $L$. Here, $\tau''$ denotes the partial ordering of $\NN^m$ by componentwise comparison. In particular, $\reltype(I)\leq |\bdb|$.
\end{Theorem}

\begin{proof}
    In view of Proposition \ref{prop:more-generators-m}, we may take arbitrary $F:=\calP(W^{|\bdc|},\bdX^{\bdc}) \in \Gamma_1\setminus\Gamma_3$.  If we introduce $\bdc':=(c_1',\dots,c_m')$ where $c_i'=\min(b_i,c_i)$ for each $i$,  then $\bdc'\lneq_{\tau''} \bdc$.  In the following, we want to reduce $F$ by $F':=\calP(W^{|\bdc'|},\bdX^{\bdc'})\in \Gamma_3$.  For simplicity, we will write $F= AW^{|\bdc|}-B\bdX^{\bdc}$ and $F'= A'W^{|\bdc'|}-B'\bdX^{\bdc'}$ with monomials $A,A',B,B'\in R$.

    We claim that $B'$ divides $B$ in $R$. To see this, for each $i$, we show that the factors  of $B'$ in $T_i$ divides the factors  of $B$ in $T_i$, i.e.,
    \[
        \frac{T_{i}^{b_i|\bdc'|}}{\gcd(T_{i}^{b_i|\bdc'|}, T_i^{c_i'a_i})} 
        \quad \text{divides} \quad
        \frac{T_{i}^{b_i|\bdc|}}{\gcd(T_{i}^{b_i|\bdc|}, T_i^{c_ia_i})},
    \]
    or equivalently,
    \begin{equation}
        \max(0,b_i|\bdc'|-c_i'a_i) \le 
        \max(0,b_i|\bdc|-c_ia_i).
        \label{eqn:dd'-m}
    \end{equation}
    When $c_i\ge b_i$, then $c_i'=b_i$ and $|\bdc'|\le |\bdb|\le a_i$. Thus, the left hand side of the inequality \eqref{eqn:dd'-m} becomes $0$. In particular, the inequality \eqref{eqn:dd'-m} holds in this case. When $c_i<b_i$, then $c_i'=c_i$ and the inequality \eqref{eqn:dd'-m} becomes
    \[
        \max(0,b_i|\bdc'|-c_ia_i) \le \max(0,b_i|\bdc|-c_ia_i).
    \]
    Since $|\bdc'|\le|\bdc|$, the above inequality holds. This finishes our proof of the claim.

    Now,
    \[
        F= \frac{B}{B'} \bdX^{\bdc-\bdc'} F'
        + W^{|\bdc'|}\left( A W^{|\bdc|-|\bdc'|}-\frac{B}{B'} A'\bdX^{\bdc-\bdc'}\right).
    \]
    Since both $F$ and $F'$ belong to the defining ideal $L$, so does $A W^{|\bdc|-|\bdc'|}-\frac{B}{B'} A'\bdX^{\bdc-\bdc'}$.  But this will imply that $A W^{|\bdc|-|\bdc'|}-\frac{B}{B'} A'\bdX^{\bdc-\bdc'}$ is a multiple of $\calP(W^{|\bdc|-|\bdc'|},\bdX^{\bdc-\bdc'})$.  Notice that $|\bdc|-|\bdc'|\ge 1$ by our assumption. Thus, by an induction argument on $|\bdc|$, we see that $F\in \braket{\Gamma_0\cup \Gamma_3}$.

    Meanwhile, notice that
    \[
        |\bdc'|=\deg_W \left( \ini \left( 
                \frac{B}{B'} \bdX^{\bdc-\bdc'} F'
            \right) \right) < |\bdc|=\deg_W (\ini(F)).
    \]
    Thus, 
    \[
        \ini(F)= W^{|\bdc'|}\cdot \ini \left( A W^{|\bdc|-|\bdc'|}-\frac{B}{B'} A'\bdX^{\bdc-\bdc'}\right),
    \]
    which in turn will be a multiple of $\ini(\calP(W^{|\bdc|-|\bdc'|},\bdX^{\bdc-\bdc'}))$. But again by induction, the latter is divisible by the leading monomial of some binomial in $\Gamma_0\cup \Gamma_3$. Thus, this completes our proof.  
\end{proof}

Notice that the assumption in Theorem \ref{prop:smaller-generators-m} holds particularly for the equi-generated case. Indeed, from now on, we will always assume that the ideal $I$ is equi-generated.

Now, we take a second look at the $\ZZ$-graded defining ideal $L=\bigoplus_{i\ge 1}L_i$, with $L_i$ being the degree $i$ piece of $L$. In particular, the ideal of the linear part is
\[
    L_1S= \braket{\calP(X_i,X_j)\mid 1\le j<i \le m+1}.
\]
It is known by \cite[Proposition 1.1]{MR2153889} that the defining ideal
\begin{equation}
    L= L_1S:(T_1T_2\cdots T_m)^{\infty}.
    \label{eqn:L-saturation}
\end{equation}
To see this, we also need to notice that $I$ is a monomial ideal and consequently $(T_1T_2\cdots T_{m})^{k}\in I$ for any $k\in \NN$ sufficiently large. Note that the above formula is also established in \cite[Theorem 2.1]{arXiv:1806.08184} for the multi-Rees algebra case. 

With the equation \eqref{eqn:L-saturation}, a natural question is to find a satisfactory small $\ell\in \NN$ such that
\[
    L=L_1S:(T_1T_2\cdots T_m)^{\ell}.
\]
Indeed, since $\calR(I)$ is almost Cohen--Macaulay, we have $L=L_1S:\alpha^{\infty}=L_1S:\frakm^{\infty}$ for any nonzero $\alpha\in I$ or $\alpha\in I_1(\varphi)$; cf.~\cite[Section 2.1]{zbMATH06113058}. Here, $I_1(\varphi)$ is the ideal of entries of the minimal presentation matrix of $I$ and $\frakm$ is the graded maximal ideal of $R$.

Recall that a sub-ideal $J$ is called a \emph{reduction} of the ideal $I$ if there is an integer $n$ such that $I^{n-1}J=I^n$. The \emph{reduction number} of $I$ with respect to the reduction ideal $J$ is the smallest integer $n$ such that  $I^{n-1}J=I^n$, and will be denoted by $\red_J(I)$.
Meanwhile, a reduction ideal of $I$ is called \emph{minimal} if it does not contain properly another reduction of $I$.
In our equi-generated case, one can check easily that $J:=\braket{T_1^{|\bdb|},T_2^{|\bdb|},\dots,T_m^{|\bdb|}}$ is a minimal reduction of the ideal $I= \braket{T_1^{|\bdb|},T_2^{|\bdb|},\dots,T_m^{|\bdb|}, \bdT^{\bdb}}$.
If $\epsilon$ is the socle degree of $R/(J:\bdT^{\bdb})$, then for $r=\epsilon+1$, we have $L=L_1S:\frakm^r$ by \cite[Theorem 2.2]{zbMATH06113058}. And more precisely, for \emph{any} $0\ne g\in \frakm^{r}$, we have $L=L_1S:g$. 

We can calculate the socle degree $\epsilon$ explicitly. It is clear that
\[
    J:\bdT^{\bdb}=\braket{T_i^{|\bdb|-b_i}\mid 1\le i\le m}.
\]
Thus, the socle degree of $R/(J:\bdT^{\bdb})$ equals $\sum_{i=1}^{m}(|\bdb|-b_i-1)=(m-1)|\bdb|-m$. Consequently, the expected \emph{secondary elimination degree} (with respect to this minimal reduction ideal $J$) is $r=(m-1)|\bdb|-m+1=(m-1)(|\bdb|-1)$.

From the above argument, we immediately have

\begin{Corollary}
    Assume that the Artinian almost complete intersection monomial ideal $I$ is equi-generated.  If $\ell\in \ZZ$ satisfies $\ell\ge \frac{(m-1)(|\bdb|-1)}{m}$, then $L_1S:(T_1T_2\cdots T_m)^{\ell}=L$. In particular, $L_1S:(T_1T_2\cdots T_m)^{|\bdb|}=L$. 
\end{Corollary}

Using the Gr\"obner basis given in Theorem \ref{prop:smaller-generators-m}, we can give another proof.

\begin{proof}
    In view of Theorem \ref{prop:smaller-generators-m}, it suffices to take arbitrary nonzero $F\in \Gamma_3\setminus\Gamma_0$ and show that $(T_1T_2\cdots T_m)^{\ell}F\in L_1S$. Write $F=AW^{|\bdc|}-B\bdX^{\bdc}$. We will use the leading terms of the following binomials in $\Gamma_0$ to reduce successively $(T_1T_2\cdots T_m)^{\ell}AW^{|\bdc|}$:
    \[
        \calP(W,X_i)=T_i^{|\bdb|-b_i}W-\prod_{j\ne i}T_j^{b_j} \cdot X_i, \qquad \text{for $1\le i\le m$}.
    \]
    At each step, although we will keep the monomial $A$ intact, the degree with respect to $W$ with drop by one, the degree with respect to $\bdX$ will increase by one, and the degree with respect to $\bdT$ will remain. Thus, we will have some monomial of the form 
    \[
        \bdT^{\bdalpha}A\bdX^{\bdc'}W^{|\bdc|-|\bdc'|}
    \]
    with $|\bdalpha|=m\ell$. Since $\ell\ge \frac{(m-1)(|\bdb|-1)}{m}$, the following inequalities cannot hold simultaneously:
    \[
        |\bdb|-b_i\ge \alpha_i+1\qquad \text{for $1\le i\le m$}.
    \]
    Without loss of generality, we may assume that $|\bdb|-b_1\le \alpha_1$. Whence, we will use the leading term of $\calP(W,X_1)$ to get the new monomial
    \[ 
        \frac{\bdT^{\bdalpha+\bdb}}{T_1^{|\bdb|}}AX_1\bdX^{\bdc'}W^{|\bdc|-|\bdc'|-1}.
    \]

    Starting from $(T_1T_2\cdots T_m)^{\ell}AW^{|\bdc|}$,
    after $|\bdc|$ steps, we will arrive at a monomial of the form
    \[
        \bdT^{\widetilde{\bdalpha}}A\bdX^{\widetilde{\bdc}}
    \]
    with $|\widetilde{\bdalpha}|=m\ell$ and $|\widetilde{\bdc}|=|\bdc|$.
    This means that 
    \[
        (T_1T_2\cdots T_m)^{\ell}F\equiv \widetilde{F}:=\bdT^{\widetilde{\bdalpha}}A\bdX^{\widetilde{\bdc}}-(T_1T_2\cdots T_m)^{\ell}B\bdX^{\bdc} \mod L_1S.
    \]
    Write $G:=\gcd(\bdX^{\widetilde{\bdc}},\bdX^{\bdc})$.
    Notice that as $\widetilde{F}\in L$, $\widetilde{F}$ is a multiple of $\calP(\bdX^{\widetilde{\bdc}}/G,\bdX^{\bdc}/G)$. But we have seen already that 
    \[
        \calP(\bdX^{\widetilde{\bdc}}/G,\bdX^{\bdc}/G)\in \braket{\calP(X_i,X_j)\mid 1\le j\le i\le m}\subset L_1S.
    \]
    And this completes the proof.
\end{proof}

\begin{Example}
    The bound $\ell\ge \frac{(m-1)(|\bdb|-1)}{m}$ above is sharp. To see this, we may take $m=3$ and $\bdb=(1,1,1)$. Using \texttt{Macaulay2} \cite{M2}, we can check directly that
    \[
        L_1S:(T_1T_2T_3)^1\subsetneq L_1S:(T_1T_2T_3)^2=L.
    \]
    Meanwhile, $\frac{(3-1)(3-1)}{3}=1.33\cdots$.
\end{Example}

Without loss of generality, we may further assume that $\gcd(\bdb):=\gcd(b_1,\dots,b_m)=1$; see also \cite[Lemma 3.1]{MR3295062}.  Under this assumption, one can check with ease that $\calP(W^{|\bdc|},\bdX^{\bdc})=W^{|\bdc|}-\bdX^{\bdc}$ if and only if $\bdc$ is a multiple of $\bdb$. In particular, $\calP(W^{|\bdb|},\bdX^{\bdb})$, which is the \emph{elimination equation} of $I$, belongs to the reduced Gr\"obner basis and minimal generating set of $L$; see also Proposition \ref{prop:min-basis}. This binomial has the largest possible degree with respect to $\widetilde{\bdX}$ by Theorem \ref{prop:smaller-generators-m}. This element is also the only non-zero generator of the defining ideal (toric ideal) of the special fiber algebra, $\mathcal{F}(I)$, associated to the monomial ideal $I$.  

Since we have seen in Theorem \ref{thm:ACM} that $\calR(I)$ is almost Cohen--Macaulay, by \cite[Corollary 3.9]{MR3096902}, we will have 

\begin{Corollary}
    \label{rmk:relation-type}
    If the Artinian almost complete intersection monomial ideal $I$ is equi-generated and $\gcd(\bdb)=1$, then $\dim(\mathcal{F}(I))=m$ and $\reltype(I)=\red_J(I)+1=|\bdb|$. 
\end{Corollary}

In particular, this implies that $\red_{J}(I)=|\bdb|-1$. But of course, this last formula is not difficult to derive directly.  

Furthermore, the Castelnuovo--Mumford regularity and the multiplicity of $\calR(I)$ are also very important algebraic and geometric invariants. Recall that for a standard graded algebra $A$ over a commutative ring, one denotes by $\H_{A_+}^{i}(A)$ the $i$-th local cohomology module of $A$ with respect to the graded ideal $A_+$ of elements of positive degree. The invariant $a_i(A)$ is defined to be $\max\Set{n|\H_{A_{+}}^i(A)_{n}\ne 0}$ with the convention that $a_i(A)=-\infty$ if $\H_{A_+}^i(A)=0$. Then, the \emph{Castelnuovo--Mumford regularity} of $A$ is defined to be 
\[
    \reg(A)\coloneqq \max\Set{a_i(A)+i\mid i\ge 0}.
\]
Under some conditions, this index indicates the complexity of graded structures over a polynomial ring.  

Meanwhile, suppose that $A$ is a Noetherian ${}^*$local algebra such that the unique graded maximal ideal $\frakM$ of $A$ is maximal in the ordinary sense. Then, there exists a polynomial $P_{\frakM}(n)$ of degree $d=\dim(A)$ such that $P_{\frakM}(n)=\ell_{A}(A/\frakM^n)$ for large $n\in \NN$. It is well-known that the leading term of this polynomial takes the form $\frac{\rme(A)}{d!}n^d$ for some positive integer $\rme(A)$.  The integer $\rme(A)$ will be called the \emph{multiplicity} of $A$. Under certain conditions, it can be interpreted as the volume of a manifold, or the degree of a closed subscheme of projective space.

Notice that the Rees algebra $\calR(I)$ being almost Cohen--Macaulay is equivalent to the associated graded ring $\Gr_I(R)$ having depth $\ge \dim(R)-1$, by a result of Huckaba; see also \cite[Theorem 3.7(ii)]{MR3096902}. Thus, it follows from \cite[Theorem 3.5 and Proposition 3.16]{MR3096902} and \cite[Theorem 3.5]{MR3784829} that

\begin{Corollary}
    If the Artinian almost complete intersection monomial ideal $I$ is equi-generated and $\gcd(\bdb)=1$, then $\reg(\calF(I))=\reg(\calR(I))=|\bdb|-1$ and $\rme(\calR(I))=\sum_{j=0}^{m-1}|\bdb|^j$.
\end{Corollary}

\begin{Remark}
    More generally,	suppose that $|\bdb|\le a_1=a_2=\cdots=a_m=:n$. As $\calR(I)$ is almost Cohen--Macaulay, Theorem \ref{prop:smaller-generators-m} and \cite[Corollary 3.9]{MR3096902} together will imply that $\red_J(I)\le |\bdb|-1$ for the minimal reduction ideal $J=\braket{T_1^{n}-T_m^{n},T_2^{n}-T_m^{n},\dots,T_{m-1}^{n}-T_m^{n},\bdT^{\bdb}}$. 
\end{Remark}

\begin{acknowledgements}
    The authors want to express sincere thanks to Teresa Cortadellas~Ben\'{i}tez and Carlos D'Andrea for directing their interest to the paper \cite{MR3490460}, and to Ricardo Burity and Aron Simis for helpful discussions.  They also want to express heartfelt thanks to the two reviewers for the suggestions that greatly improved this paper.  The second author is partially supported by the ``Anhui Initiative in Quantum Information Technologies'' (No.~AHY150200) and the ``Fundamental Research Funds for the Central Universities''.
\end{acknowledgements}


\end{document}